\documentclass{birkjour}
\usepackage{amssymb}
\usepackage{mathptmx}

\usepackage[english]{babel}

\DeclareSymbolFont{SY}{U}{psy}{m}{n}
\DeclareMathSymbol{\emptyset}{\mathord}{SY}{'306}

\newcommand{\slv}{{\textsl{v}}}

\renewcommand{\eqref}[1]{{\rm(\ref{#1})}}

\newcommand{\bbC}{{\mathbb C}}
\newcommand{\bbR}{{\mathbb R}}
\newcommand{\bbN}{{\mathbb N}}

\newcommand{\cB}{{\mathcal B}}

\newcommand{\cG}{{\mathcal G}}

\newcommand{\conv}{\mathop{\rm conv}}

\newcommand{\rl}{{\mathit{\,\,l}}}
\newcommand{\rr}{{\mathit{\,r}}}

\newcommand{\sE}{{\sf E}}

\newcommand{\vrho}{\varrho}

\newcommand{\fA}{\mathfrak{A}}

\newcommand{\fH}{\mathfrak{H}}

\newcommand{\fL}{\mathfrak{L}}
\newcommand{\fM}{\mathfrak{M}}
\newcommand{\fN}{\mathfrak{N}}

\newcommand{\dist}{\mathop{\rm dist}}

\newcommand{\lal}{{\langle}}
\newcommand{\ral}{{\rangle}}

\newcommand{\be}{\begin{equation}}
\newcommand{\ee}{\end{equation}}

\DeclareMathOperator{\spec}{spec}

\DeclareMathOperator*{\slim}{{\mathit s}-lim}

\newcommand{\Ran}{\mathop{\mathrm{Ran}}}

\newcommand{\Dom}{\mathop{\mathrm{Dom}}}
\newcommand{\Ker}{\mathop{\mathrm{Ker}}}

\allowdisplaybreaks

\numberwithin{equation}{section}

\newtheorem{introtheorem}{Theorem}
\newtheorem{theorem}{Theorem}[section]

\newtheorem{lemma}[theorem]{Lemma}

\newtheorem{hypothesis}[theorem]{Hypothesis}
\theoremstyle{definition}

\theoremstyle{remark}
{\it}{\rm}
\newtheorem{remark}[theorem]{Remark}
\newtheorem{example}[theorem]{Example}

\begin{document}

\title[The a priori tan$\,\Theta$ theorem]
{The a priori tan\,$\Theta$ theorem for spectral\\ subspaces$^*$\footnote{$^*$\lowercase{ar}X\lowercase{iv}:1012.1569;
\textit{I\lowercase{ntegral} E\lowercase{quations and} O\lowercase{perator} T\lowercase{heory}},
DOI:\,10.1007/s00020-012-1976-6
(the journal version is available online from http://dx.doi.org/10.1007/s00020-012-1976-6)}}

\author[S. Albeverio]{Sergio Albeverio}
\address{%
Institut f\"ur Angewandte Mathematik and HCM\\
Uni\-ver\-si\-t\"at Bonn\\
Endenicher Allee 60\\
D-53115 Bonn, Germany}
\email{albeverio@uni-bonn.de}

\author[A. K. Motovilov]{Alexander K. Motovilov}
\address{%
Bogoliubov Laboratory of Theoretical Physics\\
JINR, Joliot-Cu\-rie 6\\
141980 Dubna, Moscow
Region, Russia}
\email{motovilv@theor.jinr.ru}

\begin{abstract}
Let $A$ be a self-adjoint operator on a separable Hilbert space
$\fH$. Assume that the spectrum of $A$ consists of two disjoint
components $\sigma_0$ and $\sigma_1$ such that the set $\sigma_0$
lies in a finite gap of the set $\sigma_1$. Let $V$ be a bounded
self-adjoint operator on $\fH$ off-diagonal with respect to the
partition $\spec(A)=\sigma_0\cup\sigma_1$. It is known that if
\mbox{$\|V\|<\sqrt{2}d$}, where $d=\dist(\sigma_0,\sigma_1)$, then
the perturbation $V$ does not close the gaps between $\sigma_0$ and
$\sigma_1$ and the spectrum of the perturbed operator $L=A+V$
consists of two isolated components $\omega_0$ and $\omega_1$
originating from $\sigma_0$ and $\sigma_1$, respectively. Furthermore,
it is known that if $V$ satisfies the stronger bound $\|V\|< d$ then
for the difference of the spectral projections $\sE_A(\sigma_0)$ and
$\sE_{L}(\omega_0)$ of $A$ and $L$ associated with the spectral sets
$\sigma_0$ and $\omega_0$, respectively,
the following sharp norm estimate
holds:
$$
\|\sE_A(\sigma_0)-\sE_{L}(\omega_0)\| \leq
\sin\left(\arctan\frac{\|V\|}{d}\right).
$$
In the present work we
prove that this estimate remains valid and sharp also for\,\,
\mbox{$d\leq \|V\|<\sqrt{2}d$}, which completely settles the issue.
\end{abstract}



\maketitle

\section{Introduction}
\label{SIntro}

An important problem in the perturbation theory of self-adjoint operators
is to study the variation of the spectral subspace associated with an
isolated spectral subset that is subject to a perturbation (see,
e.g., \cite{K}). Classical trigonometric estimates in subspace
perturbation problem have been established by Davis and Kahan
\cite{DK70}. For further results on subspace variation bounds
for self-adjoint operators we refer to \cite{AM2011},
\cite{GKMV2010}, \cite{Li2000}, \cite{MakS10}, \cite{MotSel} and
the references therein.

In this article we consider a self-adjoint operator $A$ on a
separable Hilbert space $\fH$, assuming that the spectrum of $A$
consists of two disjoint components $\sigma_0$ and $\sigma_1$ such
that the set $\sigma_0$ lies in a finite gap of the set $\sigma_1$.
In other words, we suppose that
\begin{equation}
\label{tant}
\overline{\conv(\sigma_0)}\cap\overline{\sigma}_1=\emptyset
\quad\text{and}\quad
\sigma_0\subset\conv(\sigma_1),
\end{equation}
where $\mathop{\mathrm{conv}}$ denotes the convex hull and
overlining means closure. The perturbations $V$ are assumed to be
bounded and off-diagonal with respect to the partition
$\spec(A)=\sigma_0\cup\sigma_1$, that is, $V$ should anticommute
with the difference $\sE_A(\sigma_{0})-\sE_A(\sigma_{1})$ of the
spectral projections $\sE_A(\sigma_{0})$ and $\sE_A(\sigma_{1})$ of
$A$ associated with the sets $\sigma_0$ and $\sigma_1$,
respectively. For the spectral disposition \eqref{tant}, it has been
proven in \cite{KMM3} (see also \cite{Tr2009,Tretter-Book}) that the
gaps between $\sigma_0$ and $\sigma_1$ remain open if the
off-diagonal self-adjoint perturbation $V$ satisfies the (sharp)
condition
\begin{equation}
\label{V2d} \| V \| < \sqrt{2}\,\,d,
\end{equation}
where $ d:=\dist(\sigma_0,\sigma_1) $ stands for the distance
between $\sigma_0$ and $\sigma_1$. Under this condition the spectrum
of the perturbed operator $L=A+V$ consists of two isolated
components $\omega_0\subset\Delta$ and
$\omega_1\subset\bbR\setminus\Delta$. Here and in the sequel,
$\Delta$ denotes the finite gap of $\sigma_1$ that contains
$\sigma_0$. (We recall that by a finite gap of a closed set
$\sigma\subset\bbR$ one understands an open bounded interval on
$\bbR$ that does not intersect this set but both ends of which
belong to $\sigma$.) It is worth noting that the norm
bo\-und~\eqref{V2d} is also optimal in the sense that, if it is
violated, the spectrum of $L$ in the gap $\Delta$ may be completely
empty (see \cite[Example~1.6]{KMM4}).

The goal of the present paper consists in finalizing a sharp norm
estimate on the variation of the spectral subspace
$\Ran\bigl(\sE_A(\sigma_0)\bigr)$ under off-diagonal self-adjoint
perturbations that was conjectured and partly proven in
\cite{MotSel}. Our main result is as follows.

\begin{introtheorem}
\label{Th1} Given a (possibly unbounded) self-adjoint operator $A$
on a separable Hilbert space $\fH$, assume that its spectrum
consists of two disjoint components $\sigma_0$ and $\sigma_1$
satisfying condition \eqref{tant}. Let $V$ be a bounded self-adjoint
operator on $\fH$ off-diagonal with respect to the partition
$\spec(A)=\sigma_0\cup\sigma_1$ and set $L=A+V$, $\Dom(L)=\Dom(A)$. Assume in addition that $V$
satisfies the bound \eqref{V2d} and let $\omega_0=\spec(L)\cap\Delta$.
Then the difference between the spectral projections $\sE_A(\sigma_0)$ and
$\sE_{L}(\omega_0)$ of $A$ and $L$ associated with the respective spectral sets
$\sigma_0$ and $\omega_0$ satisfies
the norm estimate
\begin{equation}
\label{edif1} \|\sE_{A}(\sigma_0)- \sE_{L}(\omega_0)\| \leq \sin
\left( \arctan \frac{\|V\|}{d}\right)
\quad\biggl(<\sqrt{\text{\footnotesize$\frac{2}{3}$}}\,\,\biggr).
\end{equation}
\end{introtheorem}

We underline that for $\|V\|<d$ the bound \eqref{edif1} was
established  in \cite{MotSel}. It was called there the \textit{A
priori} $\tan\Theta$ \textit{Theorem}. For $\|V\|=d$ this bound may
be obtained from the result of \cite{MotSel} by continuity. Having
proved Theorem~\ref{Th1} we confirm the truth of the conjecture of
\cite[Remark~5.7]{MotSel} and we thus close the gap in the
subspace perturbation problem for dispositions~\eqref{tant} which
has remained for $\|V\|/d\in (1,\sqrt{2})$. We also remark
that the {a priori} $\tan\theta$ theorem for eigenvectors
\cite[Theorem~1.1]{AlMoSel} is a simple corollary of
Theorem~\ref{Th1}.

Our proof of Theorem~\ref{Th1} is essentially based on the reduction
of the subspace perturbation problem under consideration to the
study of the operator Riccati equation
\begin{equation}
\label{Ric0}
XA_0-A_1X+XBX=B^*
\end{equation}
with $A_0=A\bigr|_{\fA_0}$, $A_1=A\bigr|_{\fA_1}$, and
$B=V\bigr|_{\fA_1}$ where $\fA_0=\Ran\bigl(\sE_A(\sigma_0)\bigr)$
and $\fA_1=\Ran\bigl(\sE_A(\sigma_1)\bigr)$. In fact, the perturbed
spectral subspace $\fL_0=\Ran\bigl(\sE_L(\omega_0)\bigr)$ is the
graph of a particular solution $X\in\cB(\fA_0,\fA_1)$ to
equation~\eqref{Ric0}. In such a case (see, e.g., \cite{KMM2})
\begin{equation}
\label{EEX}
\|\sE_{A}(\sigma_0)-\sE_{L}(\omega_0)\|=\sin
\left( \arctan\|X\|\right).
\end{equation}
Thus, having established a bound for the solution $X$ one
simultaneously obtains an estimate for the norm of the difference of
the spectral projections $\sE_{A}(\sigma_0)$ and $\sE_{L}(\omega_0)$
as well as a bound for the operator angle
\begin{equation}
\label{Thet0}
\Theta=\arctan\sqrt{X^*X}
\end{equation}
between the spectral subspaces $\fA_0$ and $\fL_0$. For the concept
of operator angle and related material we refer to \cite{KMM2} and
references therein. Note that because of \eqref{Thet0} the
operator $X$ itself is usually called the angular operator for the
pair of subspaces $(\fA_0,\fL_0)$.

By \eqref{EEX} and \eqref{Thet0}, the bound \eqref{edif1} can be
equivalently written in the form
$$
\tan\Theta\leq\frac{\|V\|}{d}
$$
which implies that under conditions \eqref{tant} and \eqref{V2d} the
norm of the operator angle between $\fA_0$ and $\fL_0$ can never
exceed the value of $\arctan\sqrt{2}\, \left(\approx 54^\circ
44'\right)$.

The present article is the third in a series of papers on a priori
$\tan\Theta$ bounds, following \cite{AlMoSel,MotSel}. Its strategy,
however, is very different from the approaches used in
\cite{AlMoSel} and \cite{MotSel}. The approach of paper
\cite{MotSel} (which was actually the first in the series) is based
on the properties of sectorial operators and on an involution
technique that works only in cases where $\Theta<\pi/4$ (also
cf.\,\cite{GKMV2010}) and the corresponding angular operators $X$ in
\eqref{EEX} are contractions. The approach of \cite{AlMoSel} only applies to
individual eigenvectors of $L$ and there is no chance to extend it
to multi-dimensional spectral subspaces. The key ingredient of
the method we use in this paper is a new identity for eigenvalues and
eigenvectors of the modulus $|X|=\sqrt{X^*X}$ of $X$ that was
found only after the articles \cite{AlMoSel} and \cite{MotSel} were
written. Here we mean the identity \eqref{Id3} of Lem\-ma~\ref{Lcr}
below which allows us to obtain a norm bound for $X$ even if $X$ is
not a contraction (see Theorem \ref{ThMainM3} and its proof).

The paper is organized as follows. In Section \ref{SecOR} we prove
the key Lem\-ma~\ref{Lcr}. Then we recall some known bounds on the
shift of the spectrum of the operator $A$ under a perturbation $V$
satisfying the more detailed (and weaker than \eqref{V2d}) condition
$\|V\|<\sqrt{d|\Delta|}$ where $|\Delta|$ stands for the length of
the gap $\Delta$. We also recall a known norm bound for the angular
operator $X$ in  \eqref{EEX} that is valid for
$\|V\|<\sqrt{d(|\Delta|-d)}$. In Section~\ref{Sec3} we employ the
identity~\eqref{Id3} to obtain an estimate for $\|X\|$ already for
$\|V\|\geq\sqrt{d(|\Delta|-d)}$ but in the special case where  $|X|$
is assumed to have an eigenvalue equal to $\|X\|$. In Section
\ref{Sec4}, this estimate for $\|X\|$ is used to prove our most
general and detailed subspace variation bound (see Theorem
\ref{ThTot}). We conclude with a proof of Theorem~\ref{Th1}.

The following notations are used thro\-ug\-h\-o\-ut the paper. By a
subspace we always un\-der\-stand a closed linear subset of a
Hilbert space. The identity operator on a subspace (or on the whole
Hilbert space) $\fM$ is denoted by $I_\fM$. If no confusion arises,
the index $\fM$ may be omitted in this notation.  The Banach space
of bounded linear operators from a Hilbert space $\fM$ to a Hilbert
space $\fN$ is denoted by $\cB(\fM,\fN)$. By $\fM\oplus\fN$ we
understand the orthogonal sum of two Hilbert spaces (or orthogonal
subspaces) $\fM$ and $\fN$. The graph
$\cG(K)=\{y\in\fM\oplus\fN\,|\,\,y=x\oplus Kx,\,\, x\in\fM\}$ of a
bounded operator $K\in\cB(\fM,\fN)$ is called the {graph subspace}
(associated with the operator $K$). By $\sE_T(\sigma)$ we always
denote the spectral projection of a self-adjoint operator $T$
associated with a Borel set $\sigma\subset\bbR$. The notation
$\vrho(T)$ is used for the resolvent set of $T$. The domain and the range of
an operator $S$ are denoted by $\Dom(S)$ and $\Ran(S)$,
respectively.

\section{Preliminaries}
\label{SecOR}

It is convenient to represent the operators under
consideration as block operator matrices. Since condition
\eqref{tant} will not always be assumed, we
first adopt a hypothesis that implies no constraints on the mutual
position of the spectra of the entries $A_0$ and $A_1$.

\begin{hypothesis}
\label{HypL} Let $\fA_0$ and $\fA_1$ be complementary orthogonal
subspaces of a separable Hilbert space $\fH$. Assume that $A$ is a
self-adjoint operator on $\fH=\fA_0\oplus\fA_1$ admitting the
block diagonal representation
\begin{align}
\label{Adiag}
A & = \left(\begin{array}{cc} A_0 & 0\\
0 & A_1
\end{array}\right), \quad \Dom(A)=\fA_0\oplus\Dom(A_1),
\end{align}
with $A_0$ a bounded self-adjoint operator on $\fA_0$ and
$A_1$ a possibly unbounded self-adjoint operator on $\fA_1$. Suppose
that $V$ is an off-diagonal bounded self-adjoint operator on $\fH$, i.e.,
\begin{align}
\label{Voff}
V & =\left(\begin{array}{cc} 0 & B\\
B^* & 0
\end{array}\right),
\end{align}
where $0\neq B\in\cB(\fA_1,\fA_0)$, and let $L=A+V$,
$\Dom(L)=\Dom(A)$, that is,
\begin{equation}
\label{Ltot}
L=\begin{pmatrix} A_0 & B \\ B^* & A_1 \end{pmatrix},
\quad \Dom(L)=\fA_0\oplus\Dom(A_1).
\end{equation}
\end{hypothesis}

Under the assumptions of Hypothesis \ref{HypL}, an operator
$X\in\cB(\fA_0,\fA_1)$ is said to be a solution of the operator
Riccati equation \eqref{Ric0} if
\begin{equation}
\label{DefSolRic}
\Ran(X)\subset\Dom(A_1)
\end{equation}
and \eqref{Ric0} holds as an operator equality on $\fA_0$ (cf., e.g.,
\cite[Definition 3.1]{AMM}). Clearly, the solution $X$, whenever it
exists, satisfies $X\neq 0$; otherwise, $X=0$ implies $B=0$ which
contradicts the hypothesis. In the following by $U$ we denote the
partial isometry in the polar decomposition $X=U|X|$ of $X$. We
adopt the convention that $U$ is extended to $\Ker(X)=\Ker(|X|)$ by
\begin{equation}
\label{conviso} U|_{\Ker(X)}=0.
\end{equation}
In such a case $U$ is uniquely defined on the whole space $\fA_0$
(see, e.g., \cite[Theorem 8.1.2]{BirSol}) and
\begin{equation}
\label{Kisom} U \text{ \ is an isometry on \ } \Ran(|X|)=\Ran(X^*).
\end{equation}
The assertion below provides us with three useful identities for eigenvalues
and eigenvectors (in case they exist) of the modulus $|X|$.

\begin{lemma}
\label{Lcr} Assume Hypothesis \ref{HypL}. Let $X\in\cB(\fA_0,\fA_1)$
be a solution to the operator Riccati equation \eqref{Ric0}. Suppose
that $|X|$ has an eigenvalue $\lambda$ ($\lambda\geq 0$) and that
$u$, $u\neq 0$, is an eigenvector of $|X|$ corresponding to this
eigenvalue, i.e. $|X|u=\lambda u$. If $U$ is the isometry from the
polar representation $X=U|X|$ of the operator $X$, then
$Uu\in\Dom(A_1)$ and the following three identities hold:
\begin{align}
\nonumber
\lambda\bigl(\|A_0u\|^2+\|B^*u\|^2-\|A_1Uu\|^2-&\|BUu\|^2\bigr)\\
\label{Id2}
=(1-\lambda^2)&\bigl(\lal A_0u,BUu\ral+\lal B^*u,A_1Uu\ral\bigr),\quad\\
\label{Id1}
\lambda\bigl(\lal A_0u,BUu\ral+\lal B^*u,A_1Uu\ral\bigr)&=\|\Lambda_0 u\|^2-\|A_0 u\|^2-\|B^*u\|^2,\\
\label{Id3}
\lambda^2\bigl(\|A_1Uu\|^2+\|BUu\|^2-\|\Lambda_0 u\|^2\bigr)
&=\|A_0u\|^2+\|B^*u\|^2-\|\Lambda_0 u\|^2,
\end{align}
where the entry
\begin{equation}\label{Lam0}
\Lambda_0 =(I+|X|^2)^{1/2}(A_0+BX) (I+|X|^2)^{-1/2}
\end{equation}
is bounded and self-adjoint on $\fA_0$.
\end{lemma}
\begin{proof}
We start with remark that if $\lambda\neq 0$ then
$Uu=\frac{1}{\lambda}U|X|u=\frac{1}{\lambda}Xu$ and, hence,
$Uu\in\Dom(A_1)$ by \eqref{DefSolRic}. For $\lambda=0$ we
have $u\in\Ker(|X|)=\Ker(X)$ and then $Uu=0\in\Dom(A_1)$ by
convention \eqref{conviso}. We also notice that for the eigenvector
$u$ of $|X|$ associated with the nonzero eigenvalue $\lambda>0$ one
automatically has $u\in\Ran(|X|)$ and, thus, in this case
the assertion \eqref{Kisom} implies $U^*Uu=u$.

First we prove the identity \eqref{Id2}. If $\lambda=0$, then
\eqref{Id2} is trivial since $Uu=0$ due to \eqref{conviso}. Suppose
that $\lambda>0$ and set
\begin{equation}
\label{xG}
x:=\left(\begin{array}{c}
u\\
Xu
\end{array}\right)=
\left(\begin{array}{c}
u\\
U|X|u
\end{array}\right)=\left(\begin{array}{c}
u\\
\lambda Uu
\end{array}\right),
\end{equation}
\begin{equation}
\label{yG}
y:=\left(\begin{array}{c}
-X^*Uu\\
Uu
\end{array}\right)=
\left(\begin{array}{c}
-|X|U^*Uu\\
Uu
\end{array}\right)=
\left(\begin{array}{r}
-|X|u\\
Uu\,\,
\end{array}\right)=
\left(\begin{array}{r}
-\lambda u\\
Uu
\end{array}\right).
\end{equation}
From $Uu\in\Dom(A_1)$ one concludes that both $x$ and $y$ belong to
$\Dom(L)$. Since $X$ is a solution to the operator Riccati equation
\eqref{Ric0}, by, e.g., \cite[Lemma 5.3 and Theorem 5.5]{AMM} the
graphs $\cG(X)$ and $\cG(-X^*)$ are reducing subspaces for the
operator matrix $L$. Clearly, $x\in\cG(X)$ and $y\in\cG(-X^*)$ which
yields $Lx\in\cG(X)$ and $Ly\in\cG(-X^*)$. Since the subspaces
$\cG(X)$ and $\cG(-X^*)$ are orthogonal to each other, we have
\begin{equation}
\label{LxLy}
\lal Lx,Ly\ral=0.
\end{equation}
Using the last equalities in \eqref{xG} and \eqref{yG}
one obtains
\begin{equation}
\label{LxLy1}
Lx=\left(\begin{array}{c}
A_0u+\lambda BUu\\
B^*u+\lambda A_1Uu
\end{array}\right)
\quad\text{and}\quad
Ly=\left(\begin{array}{c}
-\lambda A_0u + BUu\\
-\lambda B^*u + A_1Uu
\end{array}\right).
\end{equation}
Substitution of the expressions for $Lx$ and $Ly$ from \eqref{LxLy1}
into the equality \eqref{LxLy} results in the identity
\eqref{Id2}.

To prove \eqref{Id1}, we begin with the following equalities:
\begin{align}
\label{Step1}
A_0BX+BA_1X &=A_0BX+B(XA_0+XBX-B^*)\\
\label{eM2}
&=(A_0+BX)^2-A_0^2-BB^*,
\end{align}
by taking into account at the step \eqref{Step1} that, due to
\eqref{Ric0}, \,$A_1X=XA_0+XBX-B^*$. Since $Xu=U|X|u=\lambda Uu$ and $Uu\in\Dom(A_1)$,
equality \eqref{eM2} yields
\begin{align}
\label{eM3} \lambda\bigl(\lal A_0u,BUu\ral+\lal B^*u,A_1Uu\ral\bigr)
&=\lal u,(A_0+BX)^2u\ral-\|A_0u\|^2-\|B^*u\|^2.
\end{align}
Clearly,
\begin{equation}
\label{Z2T2}
(A_0+BX)^2=(I+|X|^2)^{-1/2}\Lambda_0^2(I+|X|^2)^{1/2},
\end{equation}
where $\Lambda_0$ is the bounded operator given by \eqref{Lam0}.
Since $u$ is an eigenvector of $|X|$, by \eqref{Z2T2} one obtains
\begin{align}
\nonumber
\lal u,(A_0+BX)^2u\ral&=\lal(I+|X|^2)^{-1/2}u,\Lambda_0^2(I+|X|^2)^{1/2}u\ral\\
\nonumber
&=\lal(1+\lambda^2)^{-1/2}u,\Lambda_0^2(1+\lambda^2)^{1/2}u\ral\\
\label{eM4}
&=\lal u,\Lambda_0^2u\ral.
\end{align}
That the operator $\Lambda_0$ is self-adjoint follows, e.g., from
\cite[Theorem 5.5]{AMM}. Hence, combining \eqref{eM3} and
\eqref{eM4} one arrives at \eqref{Id1}.

As for the identity \eqref{Id3}, for $\lambda=0$ it follows
immediately from \eqref{Id1}. If $\lambda>0$, then \eqref{Id3} is
obtained by combining \eqref{Id1} with \eqref{Id2}.
\end{proof}

From now on we assume the spectral disposition \eqref{tant}.
When necessary, this disposition will be described in more detail as follows.

\begin{hypothesis}
\label{HypD} Assume Hypothesis \ref{HypL}. Let
$\sigma_0=\spec(A_0)$ and $\sigma_1=\spec(A_1)$.
Suppose that an open interval $\Delta=(\gamma_{\rl},\gamma_{\rr})\subset\bbR$,\,
$\gamma_\rl<\gamma_\rr$,\, is a finite gap of the set $\sigma_1$ and
$\sigma_0\subset\Delta$. Set
$d=\dist\bigl(\sigma_0,\sigma_1\bigr)$.
\end{hypothesis}

Below we will use the following assertions obtained by using
several results proven in~\cite{KMM3}.

\begin{theorem}  \label{T-KMMa}
Assume Hypothesis \ref{HypD} and suppose that
$\|V\|<\sqrt{d|\Delta|}$.
Then:
\begin{enumerate}
\item[{\rm(i)}] The
spectrum of the block operator matrix $L$ consists of two
disjoint components $\omega_0\subset\Delta$ and
$\omega_1\subset\bbR\setminus\Delta$. In particular,
\begin{equation}
\label{incl}
\min(\omega_0)\geq \gamma_\rl+(d-r_V)\text{\, and \,} \max(\omega_0)\leq\gamma_\rr- (d-r_V),
\end{equation}
where
\begin{equation}
\label{rV}
r_V:=\|V\|\tan\left(\frac{1}{2}\arctan\frac{2\|V\|}{|\Delta|-d}\right)<d.
\end{equation}

\item[{\rm(ii)}] There is a unique solution $X\in\cB(\fA_0,\fA_1)$ to
the Riccati equation \eqref{Ric0} with the properties
\begin{equation}
\label{sigL}
\qquad\qquad\qquad\spec(A_0+BX)=\omega_0\,\text{ and }\,
\spec(A_1-B^*X^*)=\omega_1;
\end{equation}
the spectral subspaces $\fL_0=\Ran\bigl(\sE_L(\omega_0)\bigr)$ and
$\fL_1=\Ran\bigl(\sE_L(\omega_1)\bigr)$ are graph subspaces of the
form $\fL_0=\cG(X)$ and $\fL_1=\cG(-X^*)$.
\end{enumerate}
\end{theorem}

\begin{remark}
Assertion (i) of Theorem \ref{T-KMMa} follows from
\cite[Theorem~3.2]{KMM3}. Assertion (ii) is obtained by
combining \cite[Theorem 2.3]{KMM3} with an existence
and uniqueness result for the operator Riccati equation \eqref{Ric0}
established in \cite[Theorem~1\,(i)]{KMM3}.
\end{remark}

A sharp {a priori} norm estimate for the operator angle between the
subspaces $\Ran\bigl(\sE_A(\sigma_0)\bigr)$ and
$\Ran\bigl(\sE_L(\omega_0)\bigr)$ and, equivalently, for the
corresponding angular operator $X$ in \eqref{EEX} was obtained
in \cite[Theorem 5.3]{MotSel} under an assumption that is stronger
than condition $\|V\|<\sqrt{d|\Delta|}$ of Theorem~\ref{T-KMMa}. We
formulate the main statement of \cite[Theorem 5.3]{MotSel} in the
following form.

\begin{theorem}[\cite{MotSel}]
\label{T-MS} Assume Hypothesis \ref{HypD}. Assume in addition that
$$\| V \| < \sqrt{d(|\Delta| - d)}.$$ Let $X$ be the unique solution
to the Riccati equation \eqref{Ric0} with the properties
\eqref{sigL}. Then
\begin{equation}
\label{thebest} \|X\| \leq\tan\left(\frac{1}{2}\arctan
\varkappa\bigl(|\Delta|,d,\|V\|\bigr)\right) \quad\Bigl(< 1\Bigr),
\end{equation}
where $\varkappa(D,d,\slv)$ is defined for
\begin{equation}
\label{Om12}
D>0,\quad 0< d\leq \dfrac{D}{2},\quad\text{and}\quad 0\leq \slv<\sqrt{d(D-d)}
\end{equation}
by
\begin{equation*}
\varkappa(D,d,\slv):=\left\{\begin{array}{cl}
\mbox{\small$\displaystyle\frac{2\slv}{d}$} & \text{if \,} \slv \leq
\mbox{\small$\displaystyle{\frac{1}{2}}$}\sqrt{d\left(D-2d\right)},\\[4mm]
\mbox{\small$\displaystyle\frac{{\slv}D+\sqrt{d(D-d)}\,
\sqrt{(D-2d)^2+4{\slv}^2}}{2\,\bigl(d(D-d)-{\slv}^2\bigr)}$} & \text{if \,} \slv
>\mbox{\small$\displaystyle{\frac{1}{2}}$}\sqrt{d\left(D-2d\right)}.
\end{array}\right.
\end{equation*}
\end{theorem}

In the sequel, the estimating function appearing on the right-hand
side of \eqref{thebest} will be denoted by $M_1$, that is,
\begin{equation}
\label{M1ef}
M_1(D,d,\slv):=\tan\left(\frac{1}{2}\arctan
\varkappa\bigl(D,d,\slv\bigr)\right),\quad (D,d,\slv)\in\Omega_1,
\end{equation}
where $\Omega_1$ stands for the set of points $(D,d,\slv)\in\bbR^3$
with coordinates $D,$ $d$, $\slv$ satisfying \eqref{Om12}.

\begin{remark}
\label{RemM1}
Using the elementary formula
$$
\tan\left(\frac{1}{2}\arctan x\right)
=\frac{x}{1+\sqrt{1+x^2}},\,\,\, x\in\bbR,
$$
one can also write the function $M_1(D,d,\slv)$ in the algebraic form
\begin{align}
\label{M1x}
M_1(D,d,{\slv})\biggl|_{\Omega_1^{(0)}}=&\text{\small$\frac{2{\slv}}{d+\sqrt{d^2+4{\slv}^2}}$},\\
\label{M11}
M_1(D,d,{\slv})\biggl|_{\Omega_1^{(1)}}=&\text{\small$\frac{{\slv}\bigl(2{\slv}+\sqrt{(D-2d)^2+4{\slv}^2}\,\bigr)
+\sqrt{d(D-d)}\,\bigl(D-2\sqrt{d(D-d)}\,\bigr)}
{D{\slv}+\sqrt{d(D-d)}\,\sqrt{(D-2d)^2+4{\slv}^2}}$},
\end{align}
where $\Omega_1^{(0)}$ and $\Omega_1^{(1)}$ denote the corresponding complementary
parts of the set $\Omega_1$,
\begin{align*}
\Omega_1^{(0)}:=&\left\{(D,d,\slv)\in\Omega_1\,\,\biggl|\,\,\, 0\leq\slv \leq
\text{\scriptsize$\frac{1}{2}$}\sqrt{d\left(D-2d\right)}\right\},\\
\Omega_1^{(1)}:=&\left\{(D,d,\slv)\in\Omega_1\,\,\biggl|\,\,\,
\text{\scriptsize$\frac{1}{2}$}\sqrt{d\left(D-2d\right)}<\slv<\sqrt{d(D-d)}\right\}.
\end{align*}
By \eqref{M1ef} we have
\begin{equation*}
0\leq M_1(D,d,\slv)< 1 \quad \text{for any \,}(D,d,\slv)\in\Omega_1.
\end{equation*}
By representation \eqref{M11} the function $M_1(D,d,\slv)$
admits a continuous extension to the part
\begin{equation}
\label{dOmega12}
\partial\Omega_{12}:=\left\{(D,d,\slv)\in\bbR^3\,\,\bigl|\,\,\,\, D>0,\,\, 0< d\leq {D}/{2},\,\,
 \slv=\sqrt{d(D-d)} \right\}
\end{equation}
of the boundary of $\Omega_1$ where $\slv=\sqrt{d(D-d)}$. For the extended function
we keep the same notation $M_1$. One
verifies by inspection that
$M_1(D,d,\slv)=1$ for any $(D,d,\slv)\in\partial\Omega_{12}$.

Obviously, the function $M_1(D,d,\slv)$ is infinitely differentiable
within the sets $\Omega_1^{(0)}$ and $\Omega_1^{(1)}$. Furthermore,
this function and the partial derivatives $\frac{\partial
M_1(D,d,\slv)}{\partial D}$, $\frac{\partial
M_1(D,d,\slv)}{\partial d}$, and $\frac{\partial
M_1(D,d,\slv)}{\partial \slv}$ vary continuously when $(D,d,\slv)$
passes through the common border
$\partial\Omega_1^{(01)}=\Omega_1^{(0)}\cap\overline{\,\,\Omega_1^{(1)}}$
of the subsets $\Omega_1^{(0)}$ and $\Omega_1^{(1)}$. Thus, the
function $M_1$ and its derivatives $\frac{\partial M_1}{\partial
D}$, $\frac{\partial M_1}{\partial d}$, and $\frac{\partial
M_1}{\partial \slv}$ are continuous on the whole set $\Omega_1$.
\end{remark}

\section{Norm bound for the angular operator in a special case}
\label{Sec3}

Technically, this section is central in the paper. We aim at
obtaining a norm bo\-und for the angular operator $X$  under
condition $\sqrt{d(|\Delta|-d)}\leq\|V\|<\sqrt{d|\Delta|}$ which is
outside of the scope of Theorem~\ref{T-MS}. In the proof we restrict
ourselves, however, to the special case where the mo\-du\-lus $|X|$
of $X$ has an eigenvalue coinciding with its norm
$\bigl\||X|\bigr\|=\|X\|$.

In order to formulate the result we introduce another estimating function
\begin{equation}
\label{M2ef}
M_2(D,d,{\slv}):=\sqrt{1+\mbox{\small$\dfrac{2{\slv}^2}{D^2}$}-
\mbox{\small$\dfrac{2}{D^2}$}\sqrt{dD-{\slv}^2}\,\sqrt{(D-d)D-{\slv}^2}},\,\,\,
(D,d,\slv)\in\Omega_2,
\end{equation}
where the set $\Omega_2$ is defined by
\begin{equation*}
\Omega_2:=\left\{(D,d,{\slv})\in\bbR^3\, \bigl|\,\, D>0,\,\, 0< d\leq D/2,\,\,
\sqrt{d(D-d)}\leq {\slv}< \sqrt{dD}\,\, \right\}.
\end{equation*}

\begin{remark}
\label{RemM2} Obviously, the function $M_2(D,d,\slv)$ is infinitely
differentiable inside $\Omega_2$ and continuous on $\Omega_2$.
One verifies by inspection that
\begin{align}
\label{M2bound}
& \min\limits_{(D,d,\slv)\in\Omega_2} M_2(D,d,\slv)=1,\,\,
\sup\limits_{(D,d,\slv)\in\Omega_2}M_2(D,d,\slv)=\sqrt{2},
\end{align}
and $M_2(D,d,\slv)=1$ for any
$(D,d,\slv)\in\partial\Omega_{12}$ where $\partial\Omega_{12}$ is
the intersection \eqref{dOmega12} of the boundaries of $\Omega_1$
and $\Omega_2$.
\end{remark}

\begin{theorem}
\label{ThMainM3}
Assume Hypothesis \ref{HypD}. Assume in addition that
\begin{equation}
\label{V2dd}
\sqrt{d(|\Delta|-d)}\leq \|V\|<\sqrt{d|\Delta|}.
\end{equation}
Let $X\in\cB(\fA_0,\fA_1)$ be the unique solution to the Riccati
equation \eqref{Ric0} with the properties \eqref{sigL}.
If $|X|$ has an eigenvalue $\mu$ such that $\mu=\|X\|$, then
the following bound holds:
\begin{equation}
\label{XMb}
\|X\|\leq M_2(|\Delta|,d,\|V\|),
\end{equation}
where the function $M_2(D,d,\slv)$ is given by \eqref{M2ef}.
\end{theorem}

\begin{proof}
Throughout the proof we assume, without loss of generality, that the
gap $\Delta$ is centered at zero, i.e.\,
$\gamma_{\rr}=-\gamma_{\rl}=\gamma$; otherwise, one replaces $A_0$
and $A_1$ by $A'_0=A_0-c I$ and $A'_1=A_1-c I$, respectively, where
$c=(\gamma_{\rl}+\gamma_{\rr})/2$ is the center of $\Delta$. The
assumption that $ \sigma_0\subset\Delta=(-\gamma,\gamma) $ and
$d=\dist(\sigma_0,\sigma_1)>0$ means that $\sigma_0\subset[-a,a]$
with $a=\gamma-d$ and $\|A_0\|=a$.

Suppose that $\mu$ is an eigenvalue of $|X|$ such that
$\mu=\|X\|=\bigl\||X|\bigr\|$  and let
$u$, $\|u\|=1$, be an eigenvector of $|X|$ associated with this
eigenvalue, i.e. $|X|u=\mu u$. If $\mu=\|X\|\leq 1$ then, under
condition \eqref{V2dd}, the bound \eqref{XMb} holds automatically by
the first equality in \eqref{M2bound}. Further on in the proof we
will always assume that $\mu>1.$

Let $\Lambda_0$ be as in \eqref{Lam0}. Since
$\spec(\Lambda_0)=\spec(A_0+BX)$, from Theorem \ref{T-KMMa} it
follows that $\spec(\Lambda_0)=\omega_0$ and then \eqref{incl}
yields
\begin{equation}
\label{L0bound}
0\leq\|\Lambda_0 u\|\leq a+r_V<\gamma,
\end{equation}
where $r_V$ is given by \eqref{rV} with $|\Delta|=2\gamma=2(a+d)$.
At the same time
$$
\|A_1Uu\|^2+\|BUu\|^2\geq \|A_1Uu\|^2\geq \gamma^2,
$$
taking into account that $u\in\Ran(|X|)$, $\|u\|=1$ and then  $\|Uu\|=1$ by
\eqref{Kisom}. Hence, by \eqref{L0bound}
$$
\|A_1Uu\|^2+\|BUu\|^2-\|\Lambda_0 u\|^2\geq \gamma^2-(a+r_V)^2>0
$$
and the identity \eqref{Id3} in Lemma \ref{Lcr} implies
\begin{equation}
\label{mu2}
\mu^2=\frac{\|A_0u\|^2+\|B^*u\|^2-\|\Lambda_0u\|^2}{\|A_1Uu\|^2+\|BUu\|^2-
\|\Lambda_0u\|^2}.
\end{equation}
Since
$$
\|A_0u\|\leq a, \quad \|A_1Uu\|\geq \gamma,\text{\, and \,} \|B^*u\|\leq\|B\|,
$$
from \eqref{mu2} it follows that
\begin{equation}
\label{mu2star}
\mu^2\leq\frac{a^2+\|B\|^2-\|\Lambda_0u\|^2}{\gamma^2+\|BUu\|^2-
\|\Lambda_0u\|^2}.
\end{equation}
Because of $\mu>1$, from \eqref{mu2star} one infers that
\begin{equation}
\label{InImp1}
a^2+\|B\|^2>\gamma^2+\|BUu\|^2.
\end{equation}
As for the quantity $\|\Lambda_0u\|$,  in view of \eqref{L0bound} we
have two options: either
\begin{equation}
\label{LlesA0}
0\leq\|\Lambda_0u\|\leq a
\end{equation}
or
\begin{equation}
\label{LgreA0}
a<\|\Lambda_0u\|\leq a+r_V.
\end{equation}
Since for any $s,t\in\bbR$ such that $t<s$ the function
$f(x):=\mbox{\small$\dfrac{s-x}{t-x}$}$ is increasing at $x<t$, in the case \eqref{LlesA0}
combining inequalities \eqref{mu2star} and \eqref{InImp1} yields
\begin{equation}
\label{CaseL}
\mu^2\leq \frac{\|B\|^2}{\gamma^2+\|BUu\|^2-
a^2}\leq \frac{\|B\|^2}{\gamma^2-a^2} \qquad (\text{if \,}\|\Lambda_0u\|\leq a).
\end{equation}

In order to treat the case \eqref{LgreA0} properly, one notices that, due to \eqref{Lam0},
\begin{align}
\nonumber
\|\Lambda_0u\|=&\|(I+|X|^2)^{1/2}(A_0+BX)(I+|X|^2)^{-1/2}u\|\\
\nonumber
\leq&\frac{\sqrt{1+\|X\|^2}}{\sqrt{1+\mu^2}}\,\|A_0u+\mu BUu\|\\
\label{EsCr}
&=\|A_0u\|+\mu\|BUu\|,
\end{align}
taking into account that $|X|u=\mu u$\, at the first step and that
$\|X\|=\mu$ at the second. Since $\|A_0u\|\leq a$, from \eqref{EsCr}
one deduces that, in the case \eqref{LgreA0},
$\|BUu\|\geq \frac{1}{\mu}\bigl(\|\Lambda_0u\|-a\bigr)>0$
and then \eqref{mu2star} implies
\begin{equation}
\label{mu2b}
\mu^2\leq\frac{a^2+\|B\|^2-\|\Lambda_0u\|^2}{\gamma^2+
\frac{1}{\mu^2}\bigl(\|\Lambda_0u\|-a\bigr)^2-
\|\Lambda_0u\|^2}\qquad(\text{if }\|\Lambda_0u\|>a).
\end{equation}
Inequality \eqref{mu2b} transforms into
\begin{equation}
\label{mu2base}
\mu^2\leq \frac{\|B\|^2+2\|\Lambda_0u\|(a-\|\Lambda_0u\|)}{\gamma^2-\|\Lambda_0u\|^2}
\qquad(\text{if }\|\Lambda_0u\|>a).
\end{equation}
By combining \eqref{CaseL}
and \eqref{mu2base} one arrives at the estimate
\begin{equation}
\label{mu2b1}
\mu^2\leq\left\{\begin{array}{cl}
\varphi(a) & \text{if \,}\|\Lambda_0u\|\leq a,\\[1mm]
\varphi(\|\Lambda_0u\|)
& \text{if \,}\|\Lambda_0u\|>a,
\end{array}\right.
\end{equation}
where the function $\varphi(z)$ for $z\in[0,\gamma)$ is defined by
\begin{align}
\label{phiz}
\varphi(z):=&\frac{\|B\|^2+2z(a-z)}{\gamma^2-z^2}.
\end{align}
One observes that $\varphi(0)=\|B\|^2/\gamma^2>0$ and
$\varphi(z)\to -\infty$ as $z\to\gamma-0$ since
$$
\|B\|^2+2\gamma(a-\gamma)=\|V\|^2-d|\Delta|<0
$$
by hypothesis \eqref{V2dd}. Again taking into account \eqref{L0bound},
by \eqref{mu2b1} one concludes that in any case
\begin{equation}
\label{mu2phi}
\mu^2\leq \max_{z\in[0,\gamma)}\varphi(z).
\end{equation}
We notice that the function \eqref{phiz} already appeared in the proof
of Lemma~2.1 in \cite{AlMoSel}. There is a single point $z_0$ within
the interval $[0,\gamma)$ (in fact, $z_0\in[0,a+r_V]$) where the derivative of this
function is zero, namely
\begin{equation}
\label{z00}
z_0=\left\{\begin{array}{cl}
0 & \text{if}\quad a=0,\\[1mm]
\mbox{\small$\dfrac{2\gamma^2-\|B\|^2}{2a}-
\sqrt{\left(\dfrac{2\gamma^2-\|B\|^2}{2a}\right)^2 -\gamma^2}$} &
\text{if}\quad a>0.
\end{array}\right.
\end{equation}
At $z_0$ the function $\varphi(z)$ attains its
maximum on $[0,\gamma)$, i.e.
\begin{equation}
\label{phimax}
\max_{z\in[0,\gamma)}\varphi(z)=\varphi(z_0).
\end{equation}
By inspection,
$\varphi(z_0)=M_2(2\gamma,\gamma-a,\|B\|)^2=M_2(|\Delta|,d,\|V\|)^2$,
where the function $M_2(D,d,\slv)$ is given by \eqref{M2ef}.
Combining this with \eqref{mu2phi} and \eqref{phimax} completes the
proof.
\end{proof}

From the two estimating functions introduced in
\eqref{M1ef} and \eqref{M2ef} we combine the total
estimating function
\begin{equation}
\label{Mtot}
M(D,d,\slv):=\left\{\begin{array}{cl}
M_1(D,d,{\slv}) & \text{if \,\,} 0\leq {\slv} <\sqrt{d(D-d)},\\[1mm]
M_2(D,d,{\slv}) & \text{if \,\,}
\sqrt{d(D-d)}\leq {\slv}< \sqrt{dD},\\
\end{array}\right.
\end{equation}
which is considered on the union $\Omega:=\Omega_1\cup\Omega_2$ of the
domains $\Omega_1$ and $\Omega_2$,
$$
\Omega=\left\{(D,d,{\slv})\in\bbR^3\, \bigl|\,\,\, D>0,\,\,\, 0< d\leq D/2,\,\,\,
0\leq {\slv}< \sqrt{dD}\,\, \right\}.
$$

\begin{remark}
\label{RemM} By Remarks \ref{RemM1} and \ref{RemM2}, the estimating
function $M(D,d,\slv)$ is continuous and uniformly bounded on the
whole set $\Omega$. It also admits a continuous extension to the
boundary $\partial\Omega$ of $\Omega$ (except for the intersection
of $\partial\Omega$ with the $D$ axis).  It should be
underlined, however, that the partial derivatives $\frac{\partial
M(D,d,\slv)}{\partial D}$, $\frac{\partial M(D,d,\slv)}{\partial
d}$, and $\frac{\partial M(D,d,\slv)}{\partial \slv}$ are
discontinuous when, for $d<D/2$, the point $(D,d,\slv)$ crosses the
common boundary $\partial\Omega_{12}$ of the sets $\Omega_1$ and
$\Omega_2$.
\end{remark}

\section{Subspace variation bound in the general case.\newline Proof of Theorem \ref{Th1}}
\label{Sec4}

The norm bound for the angular operator $X$ obtained in the previous
section for the special case where $|X|$ has an eigenvalue equal to
$\|X\|$ allows us to prove the following general subspace
perturbation bound.

\begin{theorem}
\label{ThTot} Assume Hypothesis \ref{HypD}. If
$\|V\|<\sqrt{d|\Delta|}$, then
\begin{equation}
\label{EEdif} \|\sE_A(\sigma_0)-\sE_L(\omega_0)\| \leq
\sin\bigl(\arctan M(|\Delta|,d,\|V\|)\bigr),
\end{equation}
where $\omega_0=\spec(L)\cap\Delta$ and $M(D,d,{\slv})$ is the
function defined by \eqref{Mtot}.
\end{theorem}

\begin{proof}
Assume, without loss of generality, that the gap $\Delta$
lies on the non-negative semiaxis, that is,
\begin{equation}
\label{gammas}
0\leq \gamma_{\rl}<\gamma_{\rr};
\end{equation}
otherwise, one replaces $A_0$ and $A_1$ by
$A'_0=A_0+\bigl(\frac{|\Delta|}{2}-c\bigr)I$ and
$A'_1=A_1+\bigl(\frac{|\Delta|}{2}-c\bigr)I$, respectively, where
$c=(\gamma_{\rl}+\gamma_{\rr})/2$.

First, we consider the case where the spectral subspace $\fA_0$ is
finite-dimensional. Theorem~\ref{T-KMMa}\,(ii) ensures the existence
of a unique angular operator $X$ for the pair of subspaces
$\fA_0=\Ran\bigl(\sE_A(\sigma_0)\bigr)$ and
$\fL_0=\Ran\bigl(\sE_L(\omega_0)\bigr)$. Since $\dim(\fA_0)<\infty$,
the operator $X$ is of finite rank and so is its modulus $|X|$. Then
there is an eigenvalue $\mu$ of $|X|$ such that
$\mu=\bigl\||X|\bigr\|=\|X\|$. Hence, for
$\sqrt{d(|\Delta|-d)}\leq\|V\|<\sqrt{d|\Delta|}$,  the bound
\eqref{EEdif} follows by \eqref{EEX} and \eqref{Mtot} from
Theorem~\ref{ThMainM3}. For $\|V\|<\sqrt{d(|\Delta|-d)}$ this bound
is implied by Theorem~\ref{T-MS}. Therefore, for the case where
$\dim(\fA_0)<\infty$, the bound \eqref{EEdif} has been proven.

\tolerance 10000
If the subspace $\fA_0$ is infinite-dimensional, let
$\{P_n^{(0)}\}_{n\in\bbN}$ be a sequence of finite-dimensional
orthogonal projections in $\fA_0$ such that
$\Ran\bigl(P_n^{(0)}\bigr)\subset\fA_0$ and
$\slim\limits_{n\to\infty}P_n^{(0)}=I_{\fA_0}$. Using the
projections $P_n^{(0)}$ we introduce the block diagonal operator
matrices
$$
A_n=\left(\begin{array}{cc}
P_n^{(0)} A_0 P_n^{(0)} & 0 \\
0 & A_1
\end{array}
\right), \quad \Dom(A_n):=\Dom(A)\quad\bigl(=\fA_0\oplus\Dom(A_1)\bigr),
$$
which represent the corresponding truncations of the operator $A$
with finite rank parts in $\fA_0$. We also introduce the finite
rank operators
$$
V_n=\left(\begin{array}{cc}
0 & P_n^{(0)} B   \\
B^* P_n^{(0)} & 0
\end{array}
\right)
$$
and set $L_n=A_n+V_n$, $\Dom(L_n):=\Dom(A_n)=\Dom(A)$. The operators
$A_n$ and $V_n$ are self-adjoint. Hence, so are the operators $L_n$.

Obviously, for any $\lambda\in\bbC\setminus\bbR$ the following
operator identities hold:
\begin{align}
\label{SchA}
(A_n -\lambda I)^{-1}-(A-\lambda I)^{-1}&=(A_n -\lambda I)^{-1}S_n(A-\lambda I)^{-1},\\
\label{SchL}
(L_n -\lambda I)^{-1}-(L-\lambda I)^{-1}&=(L_n -\lambda I)^{-1}(S_n+V-V_n)(L-\lambda I)^{-1},
\end{align}
where $S_n$ is the bounded operator on $\fH$ given by
\begin{equation}
\label{Sn}
S_n=\left(\begin{array}{cc}
A_0-P_n^{(0)} A_0 P_n^{(0)} & 0   \\
0 & 0
\end{array}
\right).
\end{equation}
By, e.g., \cite[Theorem 2.5.2]{BirSol} we have
$\slim\limits_{n\to\infty} V_n=V$, \,$\slim\limits_{n\to\infty}
(P_n^{(0)} A_0 P_n^{(0)})=A_0$, and then, due to~\eqref{SchA}--\eqref{Sn},
\begin{align}
\label{slimAL}
\slim_{n\to\infty} (A_n-\lambda I)^{-1}&=(A-\lambda
I)^{-1}\text{\, and  \,}
\slim_{n\to\infty}(L_n-\lambda
I)^{-1}=(L-\lambda I)^{-1}
\end{align}
for any $\lambda\in\bbC\setminus\bbR$, which means that both sequences $\{A_n\}_{n\in\bbN}$ and
$\{L_n\}_{n\in\bbN}$ are convergent in strong resolvent sense
(see, e.g., \cite[Section VIII.7]{Reed:Simon}).

Let $\widehat{A}_n$ and $\widehat{V}_n$ denote the parts of the operators
$A_n$ and $V_n$ associated with their reducing subspace
\begin{equation}
\label{fHn}
\widehat{\fH}_n=\widehat{\fA}_0^{(n)}\oplus\fA_1,
\end{equation}
where $\widehat{\fA}_0^{(n)}=\Ran\bigl(P_n^{(0)}\bigr)$. Clearly,
the operator $\widehat{A}_n$ is block diagonal with respect to the
decomposition \eqref{fHn},
$\Dom(\widehat{A}_n)=\widehat{\fA}_0^{(n)}\oplus\Dom(A_1)$, and
$\widehat{A}_n\big|_{\fA_1}=A_1$. Further, for the spectral set
$\widehat{\sigma}_0^{(n)}:=\spec\bigl(\widehat{A}_n\big|_{\widehat{\fA}_0^{(n)}}\bigr)$
we have the inclusion
\begin{equation}
\label{sisi0}
\widehat{\sigma}_0^{(n)}\subset[\gamma_{\rl} +d,\gamma_{\rr} -d]
\end{equation}
and, thus,
\begin{equation}
\label{dnd} d_n:=\dist\biggl(
\spec\bigl(\widehat{A}_n\big|_{\widehat{\fA}_0^{(n)}}\bigr),
\spec\bigl(\widehat{A}_n\big|_{\fA_1}\bigr)\biggr)=\dist(\widehat{\sigma}_0^{(n)},\sigma_1)\geq
d.
\end{equation}
By its construction, the finite rank operator $\widehat{V}_n$ is off-diagonal with
respect to the decomposition~\eqref{fHn} and
\begin{equation}
\label{VnlV}
\|\widehat{V}_n\|\leq\|V\|.
\end{equation}
By the hypothesis we have $\|V\|<\sqrt{d|\Delta|}$. Hence, from
\eqref{VnlV} and \eqref{dnd} it follows that
$\|\widehat{V}_n\|<\sqrt{d|\Delta|}\leq \sqrt{d_n|\Delta|}$. Then
Theorem \ref{T-KMMa} (i) implies that the spectrum of
$\widehat{L}_n:=\widehat{A}_n+\widehat{V}_n$ consists of two
disjoint components $\widehat{\omega}_0^{(n)}$ and
$\widehat{\omega}_1^{(n)}$ such that
\begin{equation}
\label{sisi}
\widehat{\omega}_0^{(n)}\subset[\gamma_{\rl} +d_n-r_V^{(n)},\gamma_{\rr} -d_n+r_V^{(n)}]\subset\Delta
\quad\text{and}\quad
\widehat{\omega}_1^{(n)}\subset\bbR\setminus\Delta,
\end{equation}
where $r_V^{(n)}$ is given by
\begin{equation*}
r_V^{(n)}=\|\widehat{V}_n\|\tan\left(\frac{1}{2}\arctan\frac{2\|\widehat{V}_n\|}{|\Delta|-d_n}\right).
\end{equation*}
Since $d\leq d_n\leq \frac{|\Delta|}{2}$ and $\widehat{V}_n$
satisfies \eqref{VnlV}, one easily verifies that $ d_n-r_V^{(n)}\geq
d-r_V $ with $r_V$ given by \eqref{rV}. Therefore, from the first
inclusion in \eqref{sisi} it follows that
\begin{equation}
\label{si0n}
\widehat{\omega}_0^{(n)}\subset[\gamma_{\rl} +d-r_V,\gamma_{\rr} -d+r_V]\quad\text{for any \,}n\in\bbN.
\end{equation}
Furthermore, since the spectral subspace
$\widehat{\fA}_0^{(n)}=\Ran\bigl(\sE_{\widehat{A}_n}(\widehat{\sigma}_0^{(n)})\bigr)$
is finite-dimensional, the bound
\eqref{EEdif} applies to the spectral projections
$\sE_{\widehat{A}_n}(\widehat{\sigma}_0^{(n)})$ and
$\sE_{\widehat{L}_n}(\widehat{\omega}_0^{(n)})$:
\begin{align}
\label{EEdif1}
\|\sE_{\widehat{A}_n}(\widehat{\sigma}_0^{(n)})-\sE_{\widehat{L}_n}(\widehat{\omega}_0^{(n)})\|
&\leq \sin\bigl(\arctan M(|\Delta|,d_n,\|\widehat{V}_n\|)\bigr).
\end{align}
Observing that the function $M(D,d,{\slv})$ is monotonously
increasing as the second argument decreases and/or the third one
increases, by \eqref{dnd} and \eqref{VnlV} from \eqref{EEdif1} one
infers that
\begin{align}
\label{EEdif2}
\|\sE_{\widehat{A}_n}(\widehat{\sigma}_0^{(n)})-\sE_{\widehat{L}_n}(\widehat{\omega}_0^{(n)})\|
&\leq \sin\bigl(\arctan M(|\Delta|,d,\|V\|)\bigr).
\end{align}
Now for an arbitrary $\varepsilon$ such that $0<\varepsilon<d-r_V$
we set
$\Sigma_\varepsilon:=(\gamma_{\rl}+\varepsilon,\gamma_{\rr}-\varepsilon)$.
Obviously, by \eqref{sisi0} and \eqref{si0n}, the open interval
$\Sigma_\varepsilon$ contains both sets
$\widehat{\sigma}_0^{(n)}$ and~$\widehat{\omega}_0^{(n)}$. Hence,
$\sE_{\widehat{A}_n}(\widehat{\sigma}_0^{(n)})=\sE_{\widehat{A}_n}(\Sigma_\varepsilon)$
and
$\sE_{\widehat{L}_n}(\widehat{\omega}_0^{(n)})=\sE_{\widehat{L}_n}(\Sigma_\varepsilon)$.
Then inequality \eqref{EEdif2} may be rewritten as
\begin{align}
\label{EKMM5}
\|\sE_{\widehat{A}_n}(\Sigma_\varepsilon)-\sE_{\widehat{L}_n}(\Sigma_\varepsilon)\|
&  \leq \sin\bigl(\arctan M(|\Delta|,d,\|V\|)\bigr).
\end{align}

Clearly, the spectrum of the part
$L_n\bigr|_{\widehat{\fH}_n^\perp}$ of the operator $L_n$ associated
with its reducing subspace
$\widehat{\fH}_n^\perp=\fH\ominus\widehat{\fH}_n$ consists of the
single point zero and the same holds for the spectrum of the
restriction $A_n\bigr|_{\widehat{\fH}_n^\perp}$, i.e.
\begin{equation}
\label{spAn}
\spec\bigl(L_n\bigr|_{\widehat{\fH}_n^\perp}\bigr)=
\spec\bigl(A_n\bigr|_{\widehat{\fH}_n^\perp}\bigr)=\{0\}.
\end{equation}
By \eqref{gammas} this means that none of the sets
$\spec\bigl(L_n\bigr|_{\widehat{\fH}_n^\perp}\bigr)$ and
$\spec\bigl(A_n\bigr|_{\widehat{\fH}_n^\perp}\bigr)$ intersects the
in\-ter\-val $\Sigma_\varepsilon$. Hence,  \eqref{EKMM5} yields
\begin{align}
\label{E5m}
\|\sE_{{A}_n}(\Sigma_\varepsilon)-\sE_{L_n}(\Sigma_\varepsilon)\|
&\leq\sin\bigl(\arctan M(|\Delta|,d,\|V\|)\bigr).
\end{align}
Meanwhile, equalities \eqref{spAn} considered together with the
inclusions \eqref{sisi0} and \eqref{si0n} imply
\begin{align}
\nonumber
(\gamma_{\rl},\gamma_{\rl}+d)\subset\varrho(A_n) &\text{\, and \,}
(\gamma_{\rr}-d,\gamma_{\rr})\subset\varrho(A_n),\\
\nonumber
(\gamma_{\rl},\gamma_{\rl}+d-r_V)\subset\varrho(L_n) &\text{\, and \,}
(\gamma_{\rr}-d+r_V,\gamma_{\rr})\subset\varrho(L_n).
\end{align}
Then, from the strong resolvent convergence \eqref{slimAL} of the
sequences $\{A_n\}_{n\in\bbN}$ and $\{L_n\}_{n\in\bbN}$, it follows
(see, e.g., \cite[Theorem VIII.24]{Reed:Simon}) that for any
$\varepsilon$ such that $0<\varepsilon<d-r_V$
\begin{equation*}
\slim_{n\to\infty}\sE_{A_n}(\Sigma_\varepsilon)=\sE_A(\Sigma_\varepsilon)
\text{\, and
\,}\slim_{n\to\infty}\sE_{L_n}(\Sigma_\varepsilon)=\sE_{L}(\Sigma_\varepsilon).
\end{equation*}
Passing to the limit as $n\to\infty$ in \eqref{E5m}, one
obtains
\begin{align*}
\|\sE_A(\Sigma_\varepsilon)-\sE_L(\Sigma_\varepsilon)\|
&\leq\sin\bigl(\arctan M(|\Delta|,d,\|V\|)\bigr),
\end{align*}
which is equivalent to \eqref{EEdif} since both spectral sets
$\sigma_0$ and $\omega_0$ are subsets of the interval
$\Sigma_\varepsilon$ $\bigl($see Theorem \ref{T-KMMa}\,(i)$\bigr)$.
\end{proof}

\begin{remark}
\label{RemOpt1} The bound \eqref{EEdif} is sharp. For $\|V\|<
\sqrt{d(|\Delta|-d)}$, this has been established in \cite{MotSel}
(see \cite[Remark 5.6\,(i)]{MotSel}). For $\sqrt{d(|\Delta|-d)}\leq
\|V\|<\sqrt{d|\Delta|}$ the sharpness of \eqref{EEdif} is proven by
\cite[Remark~2.3]{AlMoSel}. For convenience of the reader, below we
reproduce the corresponding examples from \cite{{AlMoSel}}
and \cite{MotSel} that prove the optimality of the bound \eqref{EEdif}.
\end{remark}

\begin{example}[\cite{AlMoSel}]
\label{Ex-AlMoSel}
Let $\fH_0=\bbC$ and $\fH_1=\bbC^2$. Assuming that $0\leq a<\gamma$ and $b_1,b_2\geq 0$, we set
$$
A_0=a, \quad A_1=\left(\begin{array}{rc}
-\gamma & 0\\
0  & \gamma
\end{array}\right),\text{ and }
B=\left(b_1 \,\,\, b_2\right).
$$
The operators ($3\times 3$ matrices) $A$, $V$, and $L$ are defined
on $\fH=\fH_0\oplus\fH_1=\bbC^3$ by equalities \eqref{Adiag},
\eqref{Voff}, and \eqref{Ltot}, respectively. The
spectrum $\sigma_0=\{a\}$ of $A_0$ lies in the gap
$\Delta=(-\gamma,\gamma)$ of the spectrum
$\sigma_1=\{-\gamma,\gamma\}$ of $A_1$. Also notice that
$d=\dist(\sigma_0,\sigma_1)=\gamma-a$ and $|\Delta|=2\gamma$.

First, consider the case where $b_1=0$ and, thus, $\|V\|=\|B\|=b_2$.
In this case, for any $b_2\geq 0$ satisfying
$b^2_2<2\gamma(\gamma+a)$, i.e., for
$\|V\|<\sqrt{|\Delta|(|\Delta|-d)}$ the matrix $L$ has a single
eigenvalue within the interval $\Delta$; two other eigenvalues of
$L$ are in $\bbR\setminus\Delta$. For the difference of the
eigenprojections $\sE_A(\sigma_0)$ and $\sE_L(\omega_0)$ we have
\begin{equation}
\label{EEex0}
\|\sE_A(\sigma_0)-\sE_L(\omega_0)\|=
\text{\small$\frac{2\|V\|}{d+\sqrt{d^2+4\|V\|^2}}$},
\end{equation}
where, as usually, $\omega_0=\spec(L)\cap\Delta$. Since
$\sqrt{|\Delta|(|\Delta|-d)}>\mbox{\small$\displaystyle{\frac{1}{2}}$}\sqrt{d(|\Delta|-2d)}$,
equality \eqref{EEex0} proves the optimality of the bound
\eqref{EEdif} for
$$
\|V\|\leq \mbox{\small$\displaystyle{\frac{1}{2}}$}\sqrt{d(|\Delta|-2d)}
$$
(see definition \eqref{M1x} of the restriction $M\bigr|_{\Omega_1^{(0)}}=M_1\bigr|_{\Omega_1^{(0)}}$
of the function $M$ onto the set $\Omega_1^{(0)}$).

Second, assume that $b$ is a positive number such that
\begin{equation}
\label{condex}
\sqrt{\gamma^2-a^2}\leq b<\sqrt{2\gamma(\gamma-a)}
\end{equation}
and let (cf. formula \eqref{z00})
\begin{equation}
\label{z0}
z_0:=\left\{\begin{array}{cl}
0 & \text{if}\quad a=0,\\[1mm]
\mbox{\small$\dfrac{2\gamma^2-b^2}{2a}-
\sqrt{\left(\dfrac{2\gamma^2-b^2}{2a}\right)^2 -\gamma^2}$} &
\text{if}\quad a>0.
\end{array}\right.
\end{equation}
The second inequality in \eqref{condex} implies $z_0\in\Delta$.
Note that, under condition \eqref{condex}, for
$$
t:=\frac{1}{2\gamma b^2}\bigl(b^2(\gamma-z_0)+(\gamma^2-z_0^2)(a-z_0)\bigr)
$$
one has $0\leq t<1$ (actually, $t\leq 1/2$). Then set
\begin{equation}
\label{b1b2-3}
b_1:=\sqrt{1-t}\,b, \quad b_2:=\sqrt{t}\,b.
\end{equation}
If $b_1$ and $b_2$ are introduced by \eqref{b1b2-3}, then $\|V\|=b$.
Furthermore, in this case \eqref{condex} is equivalent to
\begin{equation}
\label{Mcex1}
\sqrt{d(|\Delta|-d)}\leq\|V\|<\sqrt{d|\Delta|}
\end{equation}
and the number $z_0$ given by \eqref{z0} represents the single
eigenvalue of the matrix $L$ within the interval $\Delta$, that is,
$\omega_0=\spec(L)\cap\Delta=\{z_0\}$. An explicit computation of
the only eigenvector of $L$ corresponding to the eigenvalue $z_0$
results in
\begin{equation}
\label{EEex01}
\|\sE_A(\sigma_0)-\sE_L(\omega_0)\|=
\sin\bigl(\arctan M_2(|\Delta|,d,\|V\|)\bigr).
\end{equation}
Taking into account definition \eqref{Mtot} of the function $M$,
equality \eqref{EEex01} proves the sharpness of the bound
\eqref{EEdif} for $V$ satisfying \eqref{Mcex1}.
\end{example}

\begin{example}[\cite{MotSel}]
\label{Ex-MS-5.5} Let $\fH_0=\fH_1=\bbC^2$. Assume that $A_0$,
$A_1$, and $B$ are $2\times2$ matrices given respectively by
$$
A_0=\left(\begin{array}{rc}
-a & 0\\
0  & a
\end{array}\right),\,\,
A_1=\left(\begin{array}{rc}
-\gamma & 0\\
0  & \gamma
\end{array}\right),\text{ and }
B=\left(\begin{array}{cc}
b_1 & b_2\\
b_2  & b_1
\end{array}\right),
$$
where $0\leq a<\gamma$ and $b_1,b_2\geq 0$. Let the operators
($4\times4$ matrices) $A$, $V$, and $L$ be defined on
$\fH=\fH_0\oplus\fH_1=\bbC^4$ by equalities \eqref{Adiag},
\eqref{Voff}, and \eqref{Ltot}, respectively. Clearly, the spectrum
$\sigma_0=\{-a,a\}$ of $A_0$ lies in the gap
$\Delta=(-\gamma,\gamma)$ of the spectrum
$\sigma_1=\{-\gamma,\gamma\}$ of $A_1$. We also note that
$d=\dist(\sigma_0,\sigma_1)=\gamma-a$, $|\Delta|=2\gamma$, and
$\|V\|=\|B\|=b_1+b_2$.

One verifies by inspection that the $2\times 2$ matrix
$$
X=\left(\begin{array}{rr}
\varkappa_1 & \varkappa_2\\
-\varkappa_2  & -\varkappa_1
\end{array}\right),
$$
where
\begin{align*}
\varkappa_1&=\frac{2b_1\sqrt{(\gamma+a)^2+4b_2^2}}{(\gamma+a)\sqrt{(\gamma-a)^2+4b_1^2}
+(\gamma-a)\sqrt{(\gamma+a)^2+4b_2^2}},\\
\varkappa_2&=\frac{2b_2\sqrt{(\gamma-a)^2+4b_1^2}}{(\gamma+a)\sqrt{(\gamma-a)^2+4b_1^2}
+(\gamma-a)\sqrt{(\gamma+a)^2+4b_2^2}},
\end{align*}
is a solution to the operator Riccati equation \eqref{Ric0}.
Moreover, under condition
$\|V\|<\sqrt{d|\Delta|}=\sqrt{2\gamma(\gamma-a)}$ the
spectrum of $A_0+BX$ lies in the interval
$\Delta$, while both eigenvalues of $A_1-B^*X^*$ are in
$\bbR\setminus\Delta$. Applying, e.g., \cite[Theorem 2.3]{KMM3} one
concludes that the graph subspace $\cG(X)=\{x\oplus
Xx|\,\,x\in\fH_0\}$ is the spectral subspace of $L$ associated with
the spectral set $\omega_0=\spec(L)\cap\Delta$. By \eqref{EEX} this
yields
\begin{equation}
\label{EEex}
\|\sE_A(\sigma_0)-\sE_L(\omega_0)\|=\sin\bigl(\arctan(\varkappa_1+\varkappa_2)\bigr),
\end{equation}
taking into account that $\|X\|=\varkappa_1+\varkappa_2$.

Now pick up an arbitrary $b$ satisfying
\begin{equation}
\label{Ein1}
\mbox{\small$\dfrac{1}{2}$}\sqrt{2(\gamma-a)a}<b<\sqrt{\gamma^2-a^2}
\end{equation}
and set
\begin{equation}
\label{b1b2}
b_1=\mbox{\small$\dfrac{1}{2}$}(b+\beta)\,\text{ and }\,
b_2=\mbox{\small$\dfrac{1}{2}$}(b-\beta),
\end{equation}
where
$$
\beta=\left\{\begin{array}{cl}
0 & \text{if }\, a=0,\\
\dfrac{1}{a}\left(\sqrt{\gamma^2 b^2+a^2(\gamma^2-a^2-b^2)}-\gamma b\right) & \text{if }\, a>0.
\end{array}\right.
$$
The positivity of $b_1$ is obvious. The positivity of $b_2$ is
implied by the first inequality in \eqref{Ein1}. Since
$\|V\|=\|B\|=b_1+b_2=b$, the pair of inequalities \eqref{Ein1} is
equivalent to
\begin{equation}
\label{Mcex2}
\mbox{\small$\dfrac{1}{2}$}\sqrt{d(|\Delta|-2d)}<\|V\|< \sqrt{d(|\Delta|-d)}.
\end{equation}
Evaluation of the sum $\varkappa_1+\varkappa_2$ for $b_1$ and $b_2$
given by \eqref{b1b2} reduces \eqref{EEex} to
\begin{equation}
\label{EEex1}
\|\sE_A(\sigma_0)-\sE_L(\omega_0)\|=\sin\bigl(\arctan M_{\Omega_1^{(1)}}(|\Delta|,d,\|V\|)\bigr),
\end{equation}
where
$M_{\Omega_1^{(1)}}:=M_1\bigr|_{\Omega_1^{(1)}}=M\bigr|_{\Omega_1^{(1)}}$
is the restriction \eqref{M11} of the function $M$ onto the set
$\Omega_1^{(1)}$. Therefore, equality \eqref{EEex1} proves the sharpness
of the bound \eqref{EEdif} for $V$ satisfying \eqref{Mcex2}.
\end{example}

Theorem \ref{Th1} is nothing but a corollary of Theorem \ref{ThTot}.

\begin{proof}[Proof of Theorem \ref{Th1}]
Set $\fA_0=\Ran\bigl(\sE_A(\sigma_0)\bigr)$ and
$\fA_1=\Ran\bigl(\sE_A(\sigma_1)\bigr)$. With respect to the
orthogonal decomposition $\fH=\fA_0\oplus\fA_1$ the operators $A$
and $V$ are block operator matrices of the form \eqref{Adiag} and
\eqref{Voff}, respectively. The length of the gap $\Delta$ satisfies
the estimate $|\Delta|\geq 2d$ and, hence, condition \eqref{V2d}
implies $\|V\|<\sqrt{d|\Delta|}$. Then by Theorem
\ref{ThTot} we have estimate \eqref{EEdif}. It remains to
observe that, given the values of $\|V\|$ and $d$ satisfying
\eqref{V2d},  $M(D,d,\|V\|)$ is a non-increasing function of the
variable $D$, $D\geq 2d$. For $D$ varying in the interval
$[2d,\infty)$ it attains its maximal value at $D=2d$ and this value
equals
\begin{equation}
\nonumber
\max\limits_{D: \,D\geq 2d}M(D,d,\|V\|)=M(2d,d,\|V\|=\dfrac{\|V\|}{d}.
\end{equation}
Hence, \eqref{EEdif} yields \eqref{edif1},
completing the proof.
\end{proof}
\begin{remark}
\label{Rem2Sharp} Example 2.4 in \cite{AlMoSel} (representing a
version of Example \ref{Ex-AlMoSel} for $a=0$ and $b_1=b_2=b/\sqrt{2}$,
where $b\geq 0$) shows that the bound \eqref{edif1} is sharp.
\end{remark}

\vspace*{2mm} \noindent {\bf Acknowledgments.} We thank anonymous
referees for their comments and suggestions that allowed us to
improve the manuscript. A.\,K.\,Motovilov gratefully acknowledges the
kind hospitality of the Institut f\"ur Angewandte Mathematik,
Universit\"at Bonn, where a part of this research has been
con\-ducted. This work was supported by the Deutsche
For\-sch\-ungs\-gemeinschaft (DFG), the Heisenberg-Landau Program,
and the Russian Foundation for Basic Research.


\end{document}